\documentclass[11pt,a4paper]{article}

\usepackage[a4paper,margin=1in]{geometry}
\usepackage[T1]{fontenc}
\usepackage[utf8]{inputenc}
\usepackage{lmodern}
\usepackage{microtype}
\usepackage{graphicx}
\usepackage{amsmath,amssymb}
\usepackage{booktabs}
\usepackage{enumitem}
\usepackage{xcolor}
\usepackage{float}
\usepackage{hyperref}
\hypersetup{
  colorlinks=true,
  linkcolor=blue!60!black,
  citecolor=blue!60!black,
  urlcolor=blue!60!black
}
\usepackage{array}
\newcolumntype{C}[1]{>{\centering\arraybackslash}p{#1}}
\newcolumntype{L}[1]{>{\raggedright\arraybackslash}p{#1}}

\usepackage{tikz}

\usepackage{amsmath}
\usepackage{amsthm}

\usepackage[ruled,vlined,linesnumbered]{algorithm2e}

\theoremstyle{plain}
\newtheorem{theorem}{Theorem}[section]

\newtheorem{proposition}[theorem]{Proposition}

\theoremstyle{definition}
\newtheorem{definition}[theorem]{Definition}
\newtheorem{example}[theorem]{Example}

\newtheorem{remark}[theorem]{Remark}

\title{Stuck Knots: Rigidity, Invariants, and Unsticking Distance}
\author{Ioannis Diamantis\\
\small Department of Data Analytics and Digitalisation, Maastricht University\\
\small Maastricht, The Netherlands\\
\small \texttt{i.diamantis@maastrichtuniversity.nl}}
\date{}

\begin{document}

\maketitle

\begin{abstract}
A {\it stuck knot} is a knot diagram containing designated crossings, called {\it stuck crossings}, whose incident strands are required to remain locally non-separable. These rigidity constraints restrict the allowable ambient isotopies and introduce new geometric features into the study of knot embeddings. In this paper we develop a topological framework for knots governed by such constraints. We model stuck crossings as locally rigid configurations in spatial embeddings, placing stuck knots in close relation to rigid spatial graph theory while preserving the classical over–under information and orientation of crossings. We formalize the corresponding notion of isotopy and introduce the {\it unstick move}, which releases rigidity and allows classical simplifications to occur. To detect rigid structure algebraically, we construct polynomial invariants for stuck knots, including a HOMFLYPT-type invariant and a state-sum model extending the Kauffman bracket. These invariants show that rigidity contributes independent information even when the underlying classical knot type remains fixed. We further introduce a {\it relaxed isotopy} framework and define the {\it unsticking distance}, a geometric measure quantifying the minimal number of rigidity constraints that must be released in order to relate two stuck knots. This perspective interprets stuck crossings as barriers to isotopy and highlights the role of constraint release in diagrammatic simplification. 
\end{abstract}

\begingroup
\renewcommand{\thefootnote}{}
\endgroup

\section{Introduction}

Classical knot theory studies embeddings of $S^1$ into $S^3$ up to ambient isotopy, revealing deep connections between topology, geometry and algebra; see for example Rolfsen~\cite{Rolfsen}. Over time, numerous extensions of the classical framework have emerged by modifying either the class of admissible embeddings or the local move structure of knot diagrams. Examples such as singular knots \cite{Vassiliev1990}, virtual knots~\cite{Kauffman1999}, and rigid spatial graphs~\cite{Kauffman1989} demonstrate that altering local flexibility often produces new invariants and exposes previously hidden geometric phenomena (cf.~Kauffman \cite{Kauffman1991}, Yamada~\cite{Yamada1989}).

A natural way to restrict local flexibility is to impose non-separability conditions at crossings. In this spirit, stuck knots were introduced by Bataineh \cite{Bataineh2020} as knot diagrams in which selected crossings, called {\it stuck crossings}, represent strands that are physically attached while retaining over-under information. This framework extends the setting of singular knots by incorporating directional structure at rigid vertices and was originally motivated in part by applications to RNA foldings. Several diagrammatic invariants were developed in this context, including polynomial constructions capable of detecting chirality.

Following Bataineh’s introduction, a growing body of work has begun to develop the algebraic foundations of stuck knots and links. Ceniceros, Elhamdadi, Komissar and Lahrani \cite{CenicerosElhamdadiKomissarLahrani2024} investigated the topology of RNA foldings using stuck link diagrams, introducing an algebraic structure encoding the oriented stuck Reidemeister moves together with coloring invariants and explicit computations. Subsequent work by Ceniceros, Elhamdadi, Magill and Rosario \cite{CenicerosElhamdadiMagillRosario2023} extended quandle cocycle invariants to the stuck setting by assigning Boltzmann weights at both classical and stuck crossings, yielding new polynomial invariants. More recently, Bondarenko, Ceniceros, Elhamdadi and Jones \cite{BondarenkoCenicerosElhamdadiJones2025} introduced generalized quandle polynomials and the algebraic framework of stuquandles, further enriching the invariant theory of stuck links.

\smallbreak 

The present work adopts a complementary perspective. Rather than focusing primarily on diagrammatic invariants, we develop a spatial and geometric framework for studying knots governed by rigidity constraints. In particular, we interpret stuck crossings as rigid-height vertices in spatial embeddings, introduce a notion of constraint release via unsticking, and investigate how local rigidity influences global isotopy behavior. From this standpoint, stuck knots may be viewed as embeddings whose ambient isotopies are themselves constrained by prescribed local structure. Rigidity is therefore not merely combinatorial; it acts as a geometric obstruction that can prevent classical simplification moves from occurring. Understanding how such local constraints influence global knot behavior is a primary motivation for this work.

\medskip

A key operation in this setting is the \emph{unstick move}, which releases a rigidity constraint while preserving the underlying crossing information. This move allows one to study rigidity dynamically by examining when classical isotopies become available after sufficient constraints have been removed. Complementing this operational viewpoint, we construct algebraic invariants that detect rigidity directly. We present a HOMFLYPT-type polynomial adapted to diagrams containing rigid crossings and develop a state-sum invariant extending the classical Kauffman bracket. These constructions demonstrate that rigidity contributes independent algebraic information even when the classical knot type remains fixed.

Beyond invariant construction, we introduce a relaxed equivalence framework in which rigidity may be irreversibly released. This leads naturally to the notion of \emph{unsticking distance}, a geometric measure quantifying the minimal number of constraints that must be removed in order to relate two stuck knots. The unsticking distance provides a direct interpretation of rigid crossings as barriers to isotopy and highlights the role of constraint release in diagrammatic simplification. We further interpret stuck knots within the setting of rigid spatial graphs, situating the theory inside a broader topological context and emphasizing that rigidity should be understood as part of the embedded structure rather than as a diagrammatic artifact.

\medskip

The main contributions of this paper are as follows:

\begin{itemize}
\item We establish a spatial framework for stuck knots by interpreting stuck crossings as rigid-height vertices in embedded graphs.

\item We formalize stuck isotopy and introduce the unstick move as a mechanism for releasing local constraints.

\item We construct polynomial invariants, including a HOMFLYPT-type polynomial and a state-sum model, that detect rigidity.

\item We develop a relaxed isotopy framework and define the unsticking distance, providing a geometric measure of rigidity release.

\item We relate stuck knots to rigid spatial graph theory, positioning the subject within a broader landscape of topology under local constraints.
\end{itemize}

\medskip

Taken together, these results indicate that stuck knots form a natural class of constrained embeddings whose behavior is governed by the interaction between local rigidity and global topology. The structures introduced here suggest a foundation for studying knot theory in the presence of prescribed geometric constraints. This paper should be viewed as the first step toward a theory of topology under local constraints.

\medskip

The paper is organized as follows. Section~\ref{sec:StuckKnots} introduces stuck knots, their spatial realization and defines stuck isotopy. In Section~\ref{sec:unstick} we define the unstick move. Section~\ref{sec:HOMFLYPT} develops a HOMFLYPT-type invariant, while Section~\ref{sec:KaufBr} introduces a state-sum invariant. In Section~\ref{sec:SpatialGraphs} we interpret stuck knots as rigid spatial graphs, and in Section~\ref{sec:Unsticking Distance} we introduce relaxed isotopy and the unsticking distance, providing a geometric measure of rigidity release.


\section{Stuck Knots as Rigid-Height Embeddings}\label{sec:StuckKnots}

The concept of stuck knots was introduced by Bataineh in \cite{Bataineh2020} and is closely related to the theory of rigid vertex spatial graphs developed by Kauffman \cite{Kauffman1989}.

\begin{definition}\label{def:stucKnots}
An \emph{oriented stuck knot diagram} is a pair $(D,S)$ where $D$ is an oriented knot or link diagram in the plane and $S$ is a distinguished finite subset of its crossings, called \emph{stuck crossings}. Each crossing of $D$ is therefore of one of the following two types:

\begin{itemize}
\item \emph{Classical crossing:} a usual crossing with specified over-under information.

\item \emph{Stuck crossing:} a crossing whose over–under information is designated as fixed.
\end{itemize}
\end{definition}

The height data may be recorded diagrammatically by marking one pair of opposite regions, following the regional convention used in Kauffman's state model \cite{Kauffman1987}. The regional marking forms part of the combinatorial data of the diagram, as illustrated in Figure~\ref{fig:crossing_types}).

\begin{figure}[ht]
\centering
\includegraphics[width=0.7\textwidth]{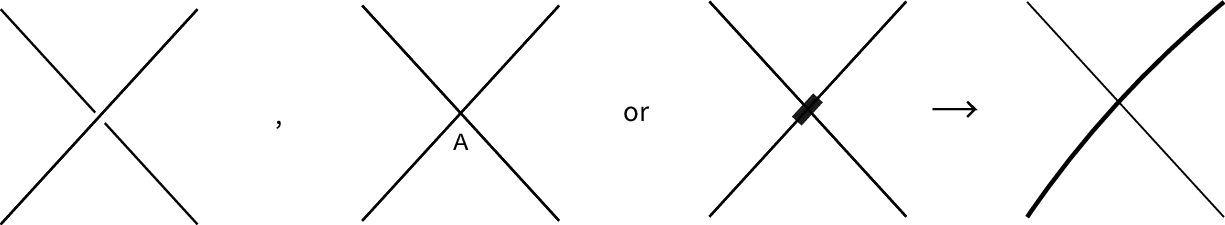}
\caption{Left: a classical crossing. Right: a stuck crossing represented as a rigid vertex together with a distinguished region (marked $A$) determining the height convention.}
\label{fig:crossing_types}
\end{figure}

\begin{remark}
Stuck knots are related to several generalizations of classical knot theory. Unlike singular knots, where double points represent unresolved crossings (cf.~Kauffman \cite{Kauffman1991}), rigid-height crossings encode fixed crossing information. They also differ from rigid-vertex spatial graphs, as the vertex is endowed with explicit height data. Finally, stuck crossings are not indeterminate as in pseudo-knot theory (see, e.g., Hanaki \cite{Hanaki2010}), but instead impose a permanent geometric constraint.
\end{remark}

Figure~\ref{fig:stuck_diagram} shows typical stuck knot diagrams.

\begin{figure}[ht]
\centering
\includegraphics[width=0.45\textwidth]{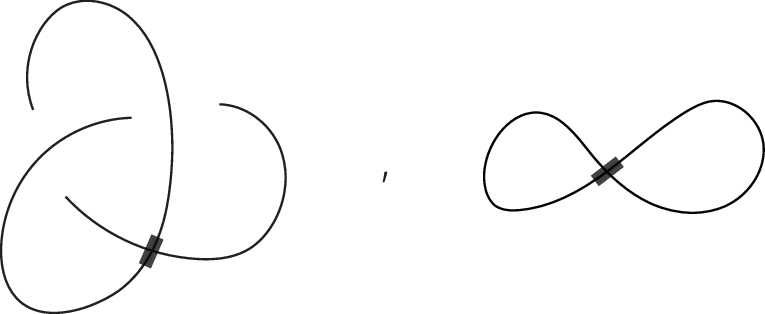}
\caption{Left: The trefoil knot containing both classical and stuck crossings. Right: The unknot with one stuck crossing.}
\label{fig:stuck_diagram}
\end{figure}

A defining feature of a stuck crossing is that the incident strands are required to remain locally incident at the vertex in the equivalence relation introduced later; in particular, the vertex cannot be removed by any move that separates the strands. To model this constraint rigorously, we interpret stuck crossings as rigid vertices in spatial graphs endowed with height ordering.

\begin{definition}
A \emph{rigid-height vertex} is a crossing arising from the transverse projection of two arcs together with a specified ordering indicating which arc passes over the other.
\end{definition}

The rigidity condition indicates that the incident arcs are to be regarded as locally non-separable, while the precise implications for admissible deformations are deferred to the definition of stuck isotopy (see \S~\ref{stuckiso}). In this way, a stuck knot may be viewed as arising from an embedding of a $4$-valent spatial graph with rigidity imposed at designated vertices. This perspective places stuck knots in close conceptual proximity to rigid vertex spatial graphs, while retaining the directional information characteristic of classical crossings. The correspondence between diagrams and spatial embeddings can be formalized through a lifting construction in the next subsection.

We emphasize that rigidity imposes structural constraints on the embedding rather than merely decorating a planar projection. Consequently, stuck knots are treated as genuine topological objects rather than purely diagrammatic constructions.


\subsection{Spatial Realization}

Although Definition~\ref{def:stucKnots} is diagrammatic, every stuck knot admits a natural realization in three-dimensional space, analogous to the lift constructions used for pseudo links \cite{Diamantis2024}. This spatial interpretation reinforces that stuck knots are topological objects rather than purely planar artifacts.

\begin{definition}
A \emph{spatial stuck knot} is a pair $(f,S)$ consisting of a tame embedding $f:S^1 \hookrightarrow S^3$ together with a finite distinguished subset $S \subset S^1$, called the set of \emph{stuck points}. For each $p \in S$, the image $f(p)$ is equipped with a prescribed local height ordering in a small neighborhood $B_\varepsilon(f(p))$, determining a fixed over-under relation between the incident arcs. These points are called \emph{rigid vertices}.

Generic planar projections of such embeddings produce diagrams in which 
transverse double points correspond to classical crossings, while the 
images of points in $S$ project to \emph{stuck crossings}.
\end{definition}

Such embeddings may be obtained via a standard lifting procedure. Each classical crossing is realized inside a sufficiently small $3$-ball so that the over arc lies above the under arc without intersecting it. Stuck crossings are supported within small $3$-balls containing rigid-height vertices, where the incident arcs intersect at a rigid vertex and respect the prescribed height ordering. The arcs connecting crossings are replaced by embedded arcs in three-space chosen to avoid unintended intersections (for an illustration see Figure~\ref{fig:spatial_stuck}).

\begin{figure}[ht]
\centering
\includegraphics[width=0.4\textwidth]{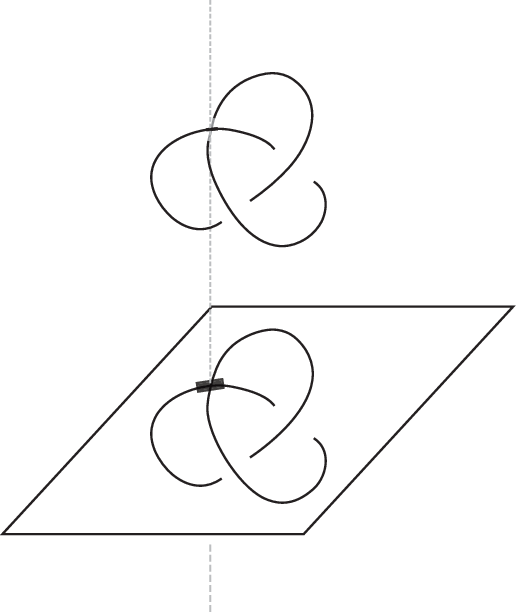}
\caption{A diagrammatic stuck knot together with a spatial realization obtained via the lifting construction.}
\label{fig:spatial_stuck}
\end{figure}

Consequently, stuck knots may be viewed as a generalization of classical knots where specific local intersection structures are preserved by rigid vertex constraints. Indeed, every classical knot admits a realization as a stuck knot with an empty set of rigid vertices. This framework therefore naturally embeds the category of classical knots into a broader class of constrained embeddings. The implications of these constraints for the global equivalence of knots will be addressed in the next subsection.

\begin{remark}
From a geometric standpoint, a stuck crossing resembles a {\it bonded interaction} between strands: one arc passes over the other while an attracting bond maintains their local attachment. Such interactions are studied in the theory of bonded knots and braids \cite{DiamantisKauffmanLambropoulou2025}, suggesting a conceptual bridge between rigidity constraints and bonded structures.
\end{remark}


\subsection{Stuck Isotopy}\label{stuckiso}

Having established the spatial interpretation of stuck knots, we now define the appropriate notion of equivalence. Since a stuck knot is realized as a spatial embedding with rigid-height vertices, the natural equivalence is ambient isotopy subject to preservation of this rigid structure.

\begin{definition}
Two spatial stuck knots are said to be \emph{stuck isotopic} if they are related by an ambient isotopy of $S^3$ that preserves each rigid-height vertex and its height ordering.
\end{definition}

In particular, no ambient isotopy is permitted that creates, removes or splits a rigid-height vertex. Thus rigidity constitutes an intrinsic feature of the embedding rather than a removable diagrammatic decoration. The spatial definition above admits a complete diagrammatic formulation via local moves, extending the classical Reidemeister theorem to embeddings with rigid vertices. Diagrammatic moves for stuck knots are illustrated in \cite{Bataineh2020} and follow from the standard arguments for spatial graphs. Here we formalize these moves within the spatial framework described above and state their completeness.

\begin{theorem}[The analogue of the Reidemeister theorem for stuck knots]
Two stuck knot diagrams represent stuck isotopic spatial embeddings if and only if they are related by a finite sequence of the local moves illustrated in Figure~\ref{fig:rigid_moves}. These moves consist of:

\begin{enumerate}
\item planar isotopies of the diagram;
\item classical Reidemeister moves performed away from stuck crossings;
\item rigid-vertex moves involving stuck crossings that preserve the height ordering and keep the incident strands attached at the vertex.
\end{enumerate}
\end{theorem}

\begin{proof}
The argument follows the classical correspondence between ambient isotopy and Reidemeister moves, with additional care taken near rigid vertices. Given an ambient isotopy between two spatial stuck knots that preserves the rigid-height vertices, a standard generic argument ensures that the isotopy may be chosen so that its projection to the plane is regular except at finitely many critical moments. Each such moment corresponds to a local configuration change modeled by one of the moves listed above.

Away from rigid vertices, the classical Reidemeister moves arise exactly as in the usual proof. Near a rigid vertex, the isotopy is constrained to preserve the transverse intersection and the prescribed height ordering, so the only admissible local changes are the rigid-vertex moves. Conversely, each diagrammatic move is supported in a sufficiently small $3$-ball and lifts to an ambient isotopy respecting the rigid structure. Therefore these moves generate precisely the equivalence induced by stuck isotopy.
\end{proof}

\begin{figure}[ht]
\centering
\includegraphics[width=0.95\textwidth]{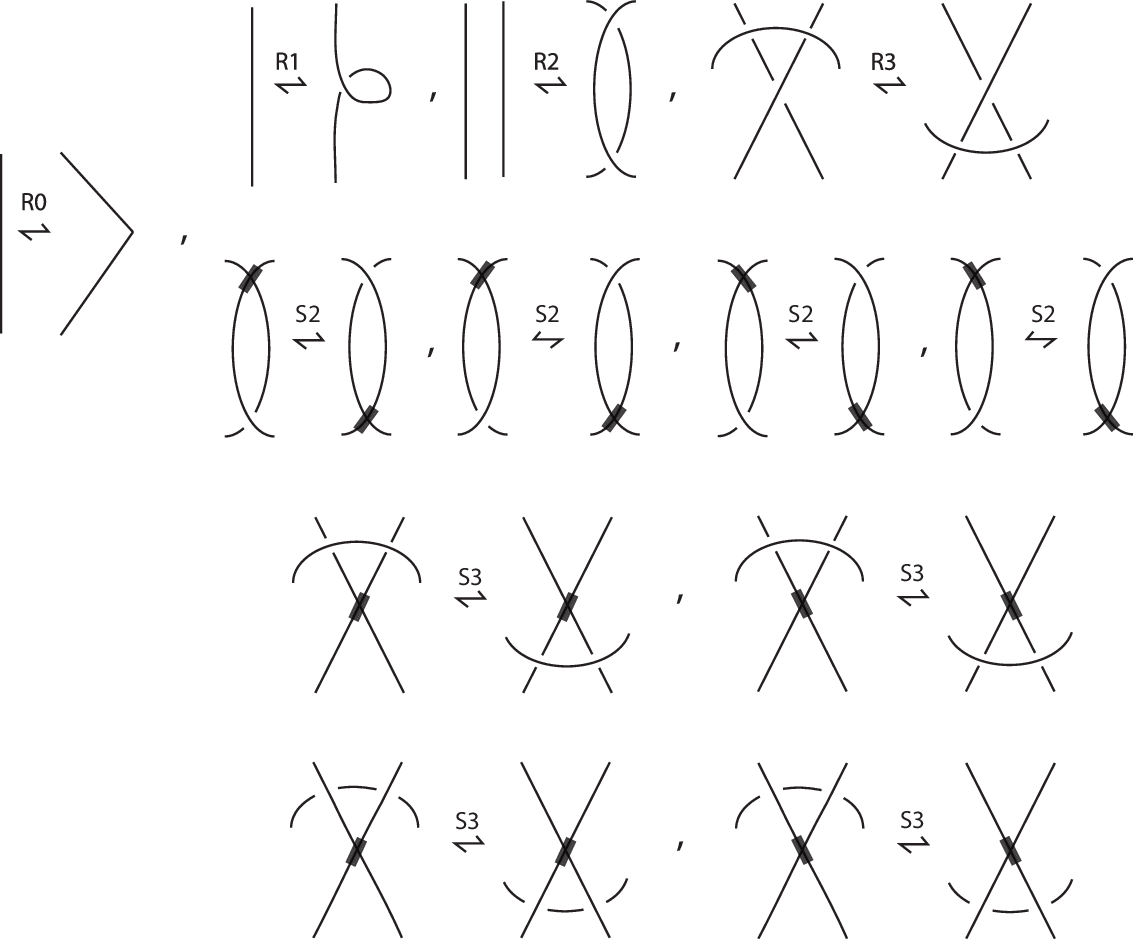}
\caption{Isotopy moves for classical and stuck crossings.}
\label{fig:rigid_moves}
\end{figure}

\begin{remark}
We highlight that Reidemeister 1 moves are not permitted at rigid-height vertices, since such moves would require separating the incident strands and thereby violate the rigidity constraint. Moreover, it is worth mentioning that oriented Reidemeister type moves for stuck links have been studied in recent work of Ceniceros et al. \cite{CenicerosElhamdadiKomissarLahrani2024}, where an algebraic framework was introduced.
\end{remark}

As mentioned before, every classical knot admits a realization as a stuck knot with no rigid-height vertices, so the classical theory appears naturally within the present framework. Under this identification, stuck isotopy restricts to classical ambient isotopy. We now formalize this relationship. Let $\mathcal{SK}$ denote the set of stuck knot (or link) types, namely spatial stuck knots modulo stuck isotopy, and let $\mathcal{K}$ denote the set of classical knot (or link) types.

\begin{definition}
The \emph{rigidity-forgetting map}
\[
\pi:\mathcal{SK}\to\mathcal{K}
\]
is defined by forgetting the rigid-height vertices in a spatial stuck knot, thereby considering only the underlying embedded circle $f(S^1) \subset S^3$.
\end{definition}

Intuitively, the map $\pi$ extracts the classical knot obtained by ignoring the rigid constraints while preserving the ambient embedding.

\begin{proposition}
The map $\pi$ is well-defined.
\end{proposition}

\begin{proof}
Suppose two spatial stuck knots represent the same element of $\mathcal{SK}$. By definition, they are related by an ambient isotopy of $S^3$ through spatial stuck knots, i.e.\ an isotopy that preserves the rigid-height vertices (and their height ordering) as part of the embedded structure. If we now forget the rigid-height data, the very same ambient isotopy is an ambient isotopy between the underlying embedded circles in $S^3$. Hence the resulting classical knot type is independent of the chosen representative, and $\pi$ is well-defined. Equivalently, the map $\pi$ is the canonical map obtained by forgetting the rigid-height vertices.
\end{proof}

The map $\pi$ shows that every stuck knot determines an underlying classical knot type. Conversely, classical knots may be viewed as special cases of stuck knots in which no rigid vertices are present. This observation leads to a natural inclusion.

\begin{definition}
Define
\[
\iota:\mathcal{K}\to\mathcal{SK}
\]
by viewing a classical knot as a stuck knot with no rigid-height vertices.
\end{definition}

Thus classical knot theory embeds as a full sub-theory of stuck knot theory.

\begin{proposition}
The map $\iota$ is well-defined and injective, and $\pi\circ\iota=\mathrm{id}_{\mathcal{K}}$.
\end{proposition}

\begin{proof}
If two classical knots are ambient isotopic, then, when regarded as stuck knots with no rigid-height vertices, the same ambient isotopy is admissible in the stuck setting; hence $\iota$ is well-defined. Moreover, since $\iota$ introduces no rigid-height vertices, applying $\pi$ forgets nothing, so $\pi\circ\iota=\mathrm{id}_{\mathcal{K}}$. Injectivity follows immediately.
\end{proof}

Together, the maps $\pi$ and $\iota$ clarify the structural relationship between the two theories: classical knots embed naturally into the broader class of stuck knots, while the rigidity-forgetting map recovers the underlying classical knot type of a constrained embedding.

\begin{remark}
For a fixed classical knot type $K$, the preimage (fiber) $\pi^{-1}(K)$ typically contains many distinct stuck knot types arising from different configurations of rigid-height vertices. Thus the forgetful map $\pi$ exhibits stuck knot theory as a refinement of classical knot theory in which additional rigidity data enriches the underlying embedding.
\end{remark}

\begin{remark}
If one forgets both the rigidity constraint and the height ordering at each rigid-height vertex, a spatial stuck knot determines a singular knot, that is, an immersed knot with finitely many transverse double points. Unlike the rigidity-forgetting map to classical knots, which preserves the embedded structure, this projection collapses the height data and records only the presence of transverse intersections. This observation places stuck knots in conceptual proximity to singular knot theory and suggests that techniques developed for singular knots may prove useful in the study of rigid-height embeddings.
\end{remark}


\section{The Unstick Move}\label{sec:unstick}

Since rigid-height vertices impose local constraints on ambient isotopy and may obstruct the simplification of a spatial embedding, a natural operation is to release such constraints while preserving the underlying curve. This leads to the unstick move, which plays a structural role in the theory by allowing transitions between embeddings of different ``rigidity levels''.

\begin{definition}
Let $(D,S)$ be a stuck knot diagram and let $x \in S$.  
The \emph{unstick move} at $x$ removes the rigidity constraint at that crossing, converting it into a classical crossing with the same over-under structure.
\end{definition}

We denote this operation by
\[
\mathcal{U}_x : (D,S) \to (D,S \setminus \{x\}).
\]

\begin{figure}[ht]
\centering
\includegraphics[width=0.65\textwidth]{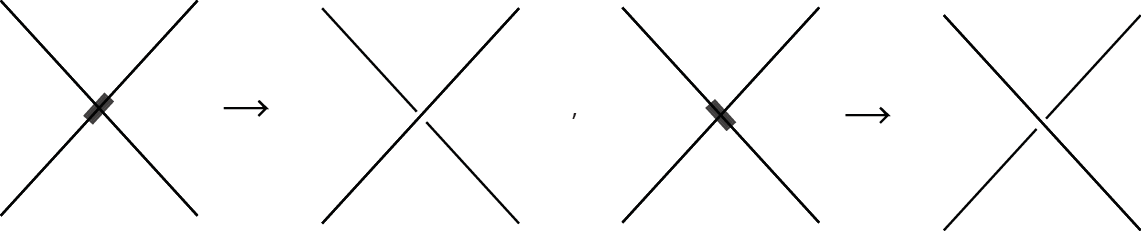}
\caption{The unstick move removes the rigidity constraint at a crossing while preserving the over-under structure.}
\label{fig:unstick_move}
\end{figure}

Geometrically, the unstick move releases a local constraint in the spatial embedding without altering the underlying embedded curve. In particular, it changes the rigidity structure rather than the ambient isotopy class of the knot. 

\begin{definition}
The \emph{trivial stuck knot} is the stuck knot represented by a planar diagram of a simple closed curve with no crossings and no rigid-height vertices; equivalently, it is the classical unknot viewed within the category of stuck knots.
\end{definition}

Thus the trivial stuck knot corresponds to the classical unknot under the inclusion of classical knots into stuck knot theory. This observation is made precise in the following proposition.

\begin{proposition}
The unstick move does not alter the underlying classical knot type.
\end{proposition}

\begin{proof}
The unstick move changes only the additional rigidity data: it replaces a rigid-height vertex by an ordinary transverse crossing while leaving the embedded circles, and hence the underlying spatial embedding obtained by the lifting construction, unchanged. Therefore applying the rigidity-forgetting map $\pi$ before or after an unstick move yields the same classical embedding, and the underlying classical knot type is unchanged.
\end{proof}

\begin{remark}
Consequently, unsticking should be viewed as modifying rigidity data within a fixed classical knot type, i.e.\ the unstick move alters the constraint structure of the embedding rather than its underlying ambient isotopy class.
\end{remark}

This observation motivates a numerical measure of rigidity defined in terms of the minimal number of unsticking operations required to reach the trivial stuck knot (see \S~\ref{sec:Unsticking Distance}).


\subsection{Rigidity Filtration}\label{subsec:filtration}

Since stuck isotopy does not permit the creation or deletion of rigid-height vertices, each stuck knot carries a well-defined measure of rigidity given by the number of such vertices.

\begin{definition}
For $k \ge 0$, let $\mathcal{R}_k$ denote the class of stuck knot types with exactly $k$ rigid-height vertices. This integer will be referred to as the {\it rigidity level} of the stuck knot.
\end{definition}

The rigidity level is well-defined at the level of stuck knot types, since stuck isotopy preserves the number of rigid-height vertices. The collection $\{\mathcal{R}_k\}_{k \ge 0}$ organizes the set of stuck knot types into levels of increasing structural constraint. Moreover, the unstick move strictly decreases the rigidity level, mapping $\mathcal{R}_k$ into $\mathcal{R}_{k-1}$. Thus rigidity endows the collection of stuck knots with a natural hierarchical structure, resembling a filtration by rigidity. This filtration suggests that stuck knot theory admits a natural grading by rigidity, a feature that provides a natural framework for algebraic and skein-theoretic constructions.

\begin{remark}
Passing from $\mathcal{R}k$ to $\mathcal{R}{k-1}$ requires an unstick move, which releases a single rigidity constraint in the spatial embedding. Note also that the level $\mathcal{R}_0$ corresponds to the class of classical knots. 
\end{remark}

\begin{example}
A diagram of the unknot containing a single rigid-height vertex belongs to $\mathcal{R}_1$. Performing an unstick move places the knot in $\mathcal{R}_0$, after which classical Reidemeister moves reduce the diagram to the trivial stuck knot.
\end{example}


\section{A Skein Theoretic Invariant for Stuck Knots}\label{sec:HOMFLYPT}

The rigidity hierarchy introduced in the previous section suggests that stuck knot theory admits a natural grading according to the number of rigid-height vertices. A central objective is therefore to construct algebraic invariants capable of detecting how such local constraints interact with the ambient topology.

Skein theory provides a powerful mechanism for converting local geometric data into computable global invariants. Since the discovery of the Jones \cite{Jones1985} and HOMFLYPT \cite{HOMFLY} polynomials, skein relations have played a foundational role in knot theory by allowing crossing information to propagate throughout a diagram. We extend this philosophy to embeddings containing rigid vertices. Unlike singular crossings, rigid crossings carry height data and cannot be removed through isotopy. They therefore behave as ``persistent local defects'' in the embedding rather than transient intersection points. Any skein invariant for stuck knots must encode this persistence while remaining compatible with the classical theory. Throughout this section, we work with oriented stuck knot diagrams and all isotopy moves are assumed to preserve orientation.

\bigbreak 

Let $L_+, L_-, L_0$ denote the standard oriented skein triple as illustrated in Figure~\ref{fig:skein_triple}. We recall that the HOMFLYPT polynomial $P$ is characterized by the skein relation
\[
aP(L_+) - a^{-1}P(L_-) = zP(L_0),
\]
together with the normalization $P(\text{unknot}) = 1$.

\begin{figure}[ht]
\centering
\includegraphics[width=0.35\textwidth]{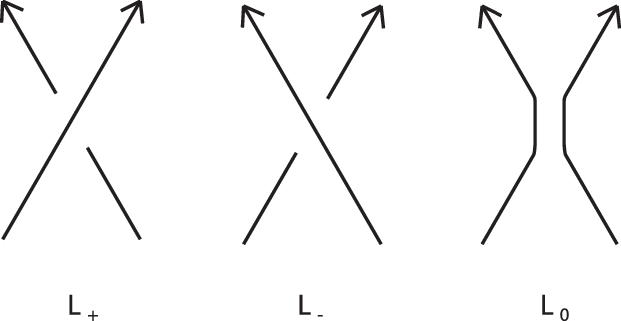}
\caption{The classical skein triple $(L_+, L_-, L_0)$.}
\label{fig:skein_triple}
\end{figure}


A rigid crossing inherits an orientation from its incident strands and appears in two forms depending on which strand is designated as passing over the other. We denote these local diagrams by $L_{\ast}^{+}$ and $L_{\ast}^{-}$.

\begin{figure}[ht]
\centering
\includegraphics[width=0.3\textwidth]{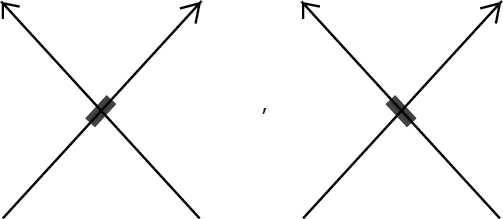}
\caption{The oriented rigid crossings $L_{\ast}^{+}$ and $L_{\ast}^{-}$.}
\label{fig:stuck_oriented}
\end{figure}

Because rigid crossings cannot be eliminated through isotopy, their skein behavior must encode the persistence of local attachment rather than merely the orientation of strands. This perspective motivates an extension of the classical HOMFLYPT recursion that incorporates rigidity directly into the skein framework. Our objective is to construct an invariant that detects the interaction between classical crossing data and rigidity constraints while remaining compatible with the rigidity hierarchy introduced earlier. We therefore augment the classical skein relation with additional local rules at rigid crossings, designed to preserve orientation data and reflect the non-separable nature of these vertices under admissible isotopies. We now introduce a skein invariant extending the HOMFLYPT polynomial to stuck knots.

\begin{definition}
The \emph{rigid HOMFLYPT polynomial} $P_R$ is defined recursively by the classical skein relation
\[
aP_R(L_+) - a^{-1}P_R(L_-) = zP_R(L_0),
\]
together with the rigid crossing relations
\[
P_R(L_{\ast}^{+}) = t\,P_R(L_0) + r\,P_R(L_+),
\]
\[
P_R(L_{\ast}^{-}) = t\,P_R(L_0) + r^{-1}P_R(L_-),
\]
where $a,z,t,r$ are independent indeterminates, and the normalization conditions

\[
P_R(\bigcirc)=1,
\qquad
P_R(L\sqcup \bigcirc)=\frac{a-a^{-1}}{z}\,P_R(L).
\]
\end{definition}

The parameters play distinct geometric roles: the variables $a$ and $z$ encode the classical skein behavior, the parameter $t$ measures the contribution of separating the strands, and the variable $r$ records the oriented height structure intrinsic to rigid crossings. The use of inverse coefficients for $r$ ensures compatibility with orientation reversal and prevents unnecessary proliferation of algebraic parameters.

\begin{remark}
When no rigid crossings are present, the defining relations reduce to those of the classical HOMFLYPT polynomial. Consequently, the restriction of $P_R$ to the stratum $\mathcal{R}_0$ recovers the classical invariant.
\end{remark}

We now verify that these recursive relations uniquely determine the invariant for all stuck knot diagrams.

\begin{proposition}
The relations above uniquely determine $P_R$ for all stuck knot diagrams.
\end{proposition}

\begin{proof}
The proof proceeds by induction on the pair $(N, s)$ under lexicographic ordering, where $N = c + s$ is the total number of crossings (the sum of classical crossings $c$ and rigid-height vertices $s$).

\begin{itemize}
    \item \textbf{Base Case:} If $N=0$, the diagram $D$ contains no crossings of any type. In this case, $D$ is an unlink. Its value is uniquely determined by the normalization $P_R(\bigcirc)=1$ and the disjoint union rule $P_R(L \sqcup \bigcirc) = \frac{a-a^{-1}}{z} P_R(L)$.

    \item \textbf{Inductive Step (Stuck crossings):} If $s > 0$, choose a rigid-height vertex $L_{\ast}^{\pm}$. Applying the rigid crossing relation expresses $P_R(L_{\ast}^{\pm})$ as a linear combination:
    \[
    P_R(L_{\ast}^{\pm}) = t\,P_R(L_0) + r^{\pm 1}P_R(L_{\pm}).
    \]
    The diagram $L_0$ has total crossings $N-1$, and the diagram $L_{\pm}$ has total crossings $N$ but with $s-1$ rigid-height vertices. In both cases, the resulting diagrams are strictly lower in the lexicographic order $(N, s)$.

    \item \textbf{Inductive Step (Classical crossings):} If $s = 0$ and $N > 0$, the diagram contains only classical crossings. We follow the standard unknotting algorithm: by choosing a base point and traversing the knot, we can express $P_R(D)$ via the classical skein relation 
    \[
    aP_R(L_+) - a^{-1}P_R(L_-) = zP_R(L_0)
    \]
    as a linear combination of an unknot and diagrams with strictly fewer crossings ($N-1$). Since $s=0$ for all diagrams in this step, they remain lower in the lexicographic order.
\end{itemize}

Since the lexicographic order on $\mathbb{N} \times \mathbb{N}$ is well-founded, the recursive process terminates in a finite number of steps. Thus, $P_R(D)$ is uniquely determined as a polynomial in $\mathbb{Z}[a^{\pm 1}, z^{\pm 1}, t, r^{\pm 1}]$.
\end{proof}

Uniqueness ensures that the recursive relations determine a function on diagrams. We now verify that this function is invariant under stuck isotopy, and therefore defines an invariant of stuck knots rather than of particular diagrammatic presentations.

\begin{theorem}
The polynomial $P_R$ is invariant under stuck isotopy.
\end{theorem}

\begin{proof}
We verify that the defining relations are preserved under the moves generating stuck isotopy.

\begin{itemize}
    \item Reidemeister 2 and 3 away from rigid crossings: On diagrams containing only classical crossings, the relations reduce to those of the HOMFLYPT polynomial, whose invariance under these moves is well known.
    \item Reidemeister 1: The normalization $P_R(L\sqcup\bigcirc)=\frac{a-a^{-1}}{z}P_R(L)$ ensures invariance under the introduction or removal of a twist, exactly as in the classical case.
    \item Moves involving rigid crossings: Suppose a diagram 
 undergoes an isotopy move involving a rigid-height vertex, such as an S2 move. As illustrated in Figure~\ref{fig:HOMS2}, applying the rigid relation $P_R(L_{\ast}^{\pm}) = t P_R(L_0) + r^{\pm 1} P_R(L_{\pm})$ reduces the diagram to a linear combination of a smoothed crossing and a classical crossing. Since both the smoothing $L_0$ and the classical crossing $L_{\pm}$ are both individually invariant under the corresponding classical Reidemeister moves, their linear combination remains invariant. An analogous argument applies to the S3 moves (an arc sliding across a rigid vertex), where the rigid relation commutes with the ambient isotopy.
\end{itemize}

Since all generating moves preserve the polynomial, $P_R$ is invariant under stuck isotopy.
\end{proof}

\begin{figure}[ht]
\centering
\includegraphics[width=0.85\textwidth]{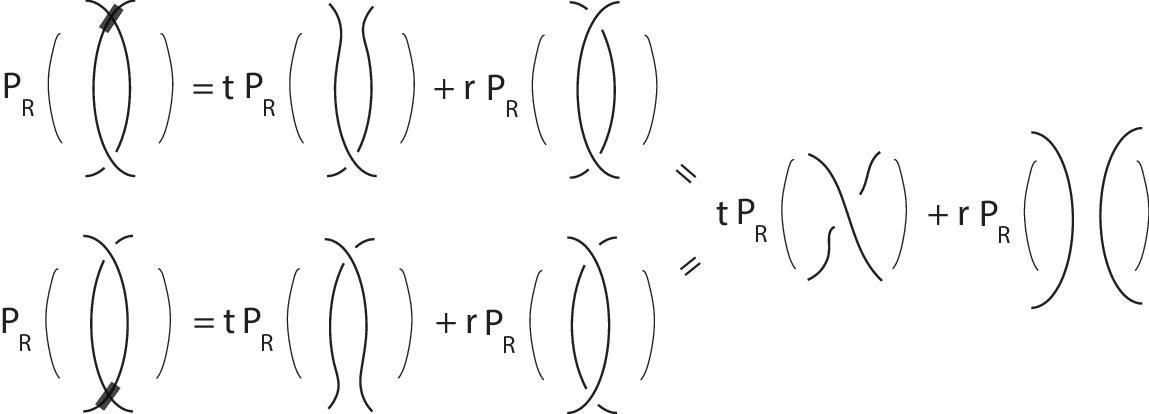}
\caption{Invariance under S2 move.}
\label{fig:HOMS2}
\end{figure}


Beyond invariance, the polynomial reflects the rigidity hierarchy. Each rigid crossing contributes independent algebraic data to the skein expansion, endowing the invariant with a natural grading by rigidity.

\begin{proposition}
Let $K^* \in \mathcal{SK}$ be a stuck knot with $k$ rigid-height vertices. Every monomial $m$ appearing in the polynomial $P_R(K^*)$ satisfies:
\[
\deg_{\{t, r\}}(m) \le k,
\]
where $\deg_{\{t, r\}}$ denotes the total degree in the variables $t$ and $r$.
\end{proposition}

\begin{proof}
We proceed by induction on $k$, the number of rigid-height vertices.
\begin{itemize}
    \item \textbf{Base Case ($k=0$):} $P_R(K^*)$ is the classical HOMFLYPT polynomial in $a$ and $z$. Since $t$ and $r$ do not appear, the degree is $0 \le 0$.
    \item \textbf{Inductive Step:} Suppose the inequality holds for $k-1$ rigid vertices. For a knot with $k$ vertices, we apply the rigid relation:
    \[
    P_R(L_{\ast}^{\pm}) = t\,P_R(L_0) + r^{\pm 1}P_R(L_{\pm}).
    \]
    By the inductive hypothesis, the monomials in $P_R(L_0)$ and $P_R(L_{\pm})$ have degree at most $k-1$. 
    \begin{itemize}
        \item Multiplying by $t$ results in a degree at most $(k-1) + 1 = k$.
        \item Multiplying by $r$ results in a degree at most $(k-1) + 1 = k$.
        \item Multiplying by $r^{-1}$ results in a degree at most $(k-1) - 1 = k-2$.
    \end{itemize}
    In all cases, the degree of any resulting monomial $m$ satisfies $\deg_{\{t, r\}}(m) \le k$.
\end{itemize}
\end{proof}

\begin{remark}
The recursive definition of $P_R$ follows a general schema known as \emph{tangle insertion}, introduced by Henrich and Kauffman \cite{HenrichKauffman2017}. In their framework, singular vertices or unknown crossings are resolved by inserting a specific set of tangles to generate invariants for singular knots and rigid spatial graphs. Our rigid relation can be viewed as a specialized instance of this schema, tailored to the unique height-ordering of stuck knots.
\end{remark}


We now present several examples illustrating how rigidity contributes information beyond classical knot type.

\begin{example}\label{curl_stuck_hom}
Let $D$ be a diagram consisting of a single positive stuck crossing $L_{\ast}^{+}$ on a trivial loop. The rigid skein relation yields
\[
P_R(D) = t\,P_R(L_0) + r\,P_R(L_+).
\]
Here $L_0$ is the smoothing of the stuck crossing, which is the $2$-component unlink, and $L_+$ is the diagram obtained by replacing the stuck crossing with a classical crossing (a positive kink). Since $L_+$ is a classical diagram of the unknot and $P_R$ agrees with the classical HOMFLYPT polynomial on diagrams without stuck crossings, we have $P_R(L_+)=P(\bigcirc)=1$. Using
\[
P_R(L\sqcup \bigcirc)=\frac{a-a^{-1}}{z}\,P_R(L),
\]
we obtain
\[
P_R(D) = t \left( \frac{a - a^{-1}}{z} \right) + r.
\]
\end{example}

Notably, the polynomial distinguishes this diagram from the trivial knot despite their classical equivalence, illustrating that rigidity contributes genuinely new information.

\begin{example}
Let $K$ be a classical knot diagram and let $K_x^\ast$ be the stuck knot obtained by declaring one crossing of $K$ to be a rigid-height vertex of type $L_{\ast}^{+}$. The rigid skein relation yields
\[
P_R(K_x^\ast) = t\,P_R(L_0) + r\,P_R(L_+),
\]
where $L_+$ is the original classical knot $K$ and $L_0$ is its smoothing at $x$. Since $t$ and $r$ are independent parameters, the resulting polynomial $P_R(K_x^\ast)$ cannot be expressed solely in terms of $a$ and $z$ unless $t=0$ and $r=1$. This demonstrates that $P_R$ detects local rigidity independently of the underlying classical knot type $\pi(K_x^\ast)$.
\end{example}

These examples indicate that rigidity introduces genuinely new degrees of freedom into the skein framework, indicating that the polynomial detects features not visible to classical invariants.

\begin{remark}
The structure of the skein relations hints at a broader algebraic framework. The skein relations underlying $P_R$ suggest that stuck links may admit a skein module formulation analogous to the classical setting. While we do not pursue this direction here, it provides a natural framework for future investigation.
\end{remark}


\section{A State-Sum Invariant for Stuck Knots}\label{sec:KaufBr}

The rigidity filtration introduced in \S~\ref{subsec:filtration} indicates that invariants for stuck knots should encode not only classical crossing data but also the persistence of rigid-height vertices. While the HOMFLYPT-type polynomial incorporates rigidity through modified skein relations, state-sum models provide a complementary approach in which rigid vertices are retained as local states within the expansion.

State-sum constructions translate local geometric structure into computable global data and play a fundamental role in knot theory. In the present setting, they offer a natural mechanism for ensuring that rigidity remains visible at every stage of the expansion. We therefore construct a bracket type invariant for stuck knots generalizing the classical Kauffman bracket \cite{Kauffman1987}. The key idea is that classical crossings admit smoothing operations, whereas rigid vertices are retained as structural states reflecting their role as locally non-separable attachments in the spatial embedding. As a result, rigidity remains visible through the state expansion rather than disappearing during resolution.

The resulting polynomial is designed to reflect both classical entanglement and the structural constraints imposed by rigid attachment.


\subsection{The State Model}

We now formalize the propagation of rigidity through a state expansion. Let $D$ be a diagram of a stuck knot $K^*$. Each classical crossing admits the usual $A$ and $B$ smoothings. At a stuck crossing we introduce a third local state representing the persistent, non-separable contact between strands.

\begin{figure}[ht]
\centering
\includegraphics[width=0.65\textwidth]{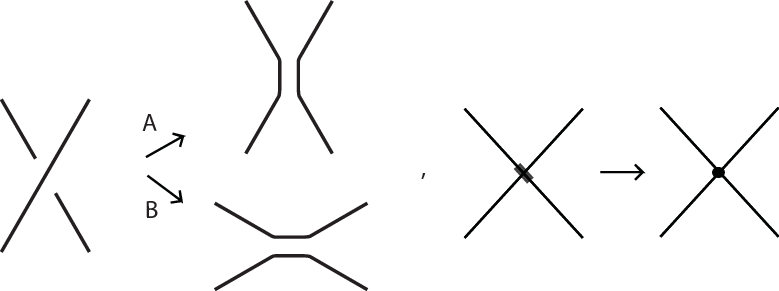}
\caption{Local states at a crossing: $A$ smoothing, $B$ smoothing, and the rigid vertex state $V$.}
\label{fig:states}
\end{figure}

Accordingly, each crossing of $D$ is resolved according to one of the following local rules:

\begin{enumerate}
\item $A$ smoothing,
\item $B$ smoothing,
\item rigid vertex state $V$ (and is the only permitted local state at stuck crossings).
\end{enumerate}

Resolving the classical crossings produces a collection of simple closed curves together with rigid vertices that persist throughout the state. Note that in the $V$ state, the stuck crossing is kept as a rigid $4$-valent vertex with the same height ordering as in $D$. 

\smallbreak 

To each local resolution we assign multiplicative weights.

\[
A\text{\, smoothing}: A, \qquad 
B\text{\, smoothing}: A^{-1}, \qquad 
V\text{\, state}: R,
\]

where $A$ and $R$ are independent indeterminates.

\begin{definition}
A \emph{state} $s$ of $D$ is a choice of local resolution at every crossing.  
Define the weight of $s$ by
\[
w(s)=A^{\alpha(s)-\beta(s)}R^{\nu(s)},
\]
where $\alpha(s)$ and $\beta(s)$ denote the numbers of $A$ and $B$ smoothings, respectively, and $\nu(s)$ counts rigid vertex states.
\end{definition}

After resolving the classical crossings, each state $s$ is realized as a spatial graph consisting of embedded arcs joined at the persistent rigid vertices. Since each rigid vertex is $4$-valent, the resulting state is a collection of one or more connected components. Let $|s|$ denote the number of these connected components.

\begin{definition}
The \emph{stuck bracket} of a diagram $D$ is defined by the state sum:
\[
\langle D \rangle_R = \sum_{s} w(s) \,\delta^{|s|-1}, \qquad \delta = -A^2 - A^{-2},
\]
where the sum ranges over all possible states $s$ of the diagram $D$.
\end{definition}

Note that since each classical crossing admits two resolutions and rigid crossings admit one, the number of states is finite.

\begin{remark}
When no rigid crossings are present, the state model reduces exactly to the classical Kauffman bracket. Thus the stuck bracket generalizes the classical Kauffman bracket by adjoining a multiplicative parameter recording the presence of rigid vertices.
\end{remark}

\begin{figure}[ht]
\centering
\includegraphics[width=0.5\textwidth]{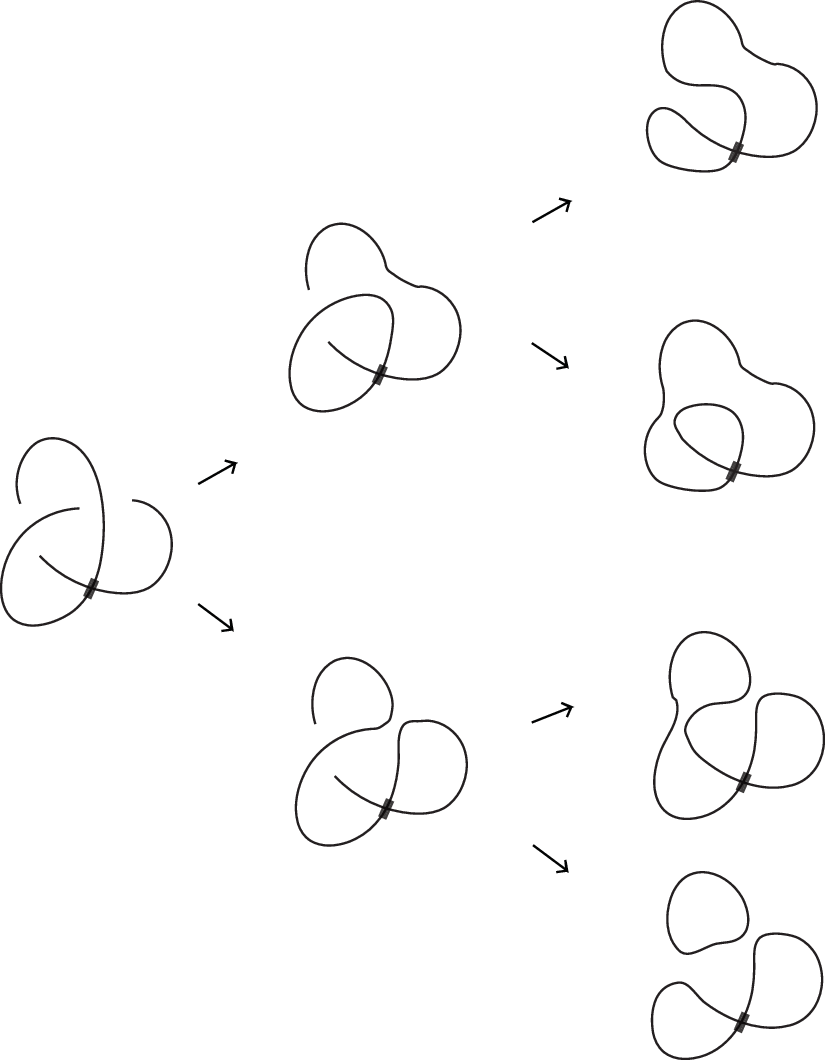}
\caption{Example of a state obtained by resolving classical crossings while retaining rigid vertices.}
\label{fig:state_example}
\end{figure}


As in the classical theory, the bracket $\langle D \rangle_R$ is invariant under Reidemeister moves 2 and 3 but fails to be invariant under Reidemeister 1. To obtain an ambient isotopy invariant, we normalize using the writhe.

Let $w(D)$ denote the writhe of $D$, counting only classical crossings. Rigid vertices contribute no writhe since they do not represent oriented crossings.

\begin{definition}
Define the polynomial
\[
P_{K^*}(A,R)
=
(-A^3)^{-w(D)}\langle D\rangle_R.
\]
\end{definition}

\begin{theorem}\label{invariant}
The polynomial $P_{K^*}(A,R)$ is invariant under stuck isotopy.
\end{theorem}

\begin{proof}
We verify invariance under the generating moves of stuck isotopy.

\begin{itemize}
\item Reidemeister 2 and 3 away from rigid vertices: These follow from the classical bracket argument \cite{Kauffman1987}, since the local state expansions are unchanged.
\item Moves involving rigid vertices: Admissible isotopies preserve rigid vertices and do not separate the incident strands. Each rigid move acts locally and preserves the rigid vertex together with its incident arcs, thereby inducing a one-to-one correspondence between states with identical weights.
\item Reidemeister 1: Since rigid vertices do not admit Reidemeister 1 moves, all twist contributions arise from classical crossings. A twist introduces the usual multiplicative factor in the bracket, which is exactly cancelled by the writhe normalization. Hence the polynomial is invariant under all allowed moves.
\end{itemize}
Hence $P_{K^*}(A,R)$ is an invariant of stuck knot types.
\end{proof}

\begin{remark}
Setting $R=1$ recovers the classical Kauffman bracket normalization, while $R=0$ emphasizes the contribution of persistent rigid vertices in the state expansion.
\end{remark}


\subsection{Properties of the Stuck Bracket}

We now examine several structural features of the stuck bracket. In particular, we show that the invariant detects rigidity, possesses a natural universality among state-sum constructions, and complements the oriented skein invariant introduced earlier.

\paragraph{Detection of Rigidity.}

The additional variable $R$ ensures that rigid vertices contribute explicitly to the invariant.

\begin{proposition}
Rigid vertices contribute explicitly to the polynomial through powers of $R$. In particular, if a diagram contains at least one rigid vertex, then $P_{K^*}(A,R)$ is generally non-constant as a polynomial in $R$, unless the $R$-dependent terms cancel identically.
\end{proposition}

\begin{proof}
Each state containing a rigid vertex contributes a factor of $R$, whereas states arising solely from classical crossings do not. Consequently, the state expansion records the presence of rigid vertices through $R$-dependent terms.
\end{proof}

\begin{example}[A classical curl vs.\ a rigid curl]\label{ex:curl_vs_rigidcurl}
Let $D_{\mathrm{cl}}$ be the standard diagram of the unknot with a single \emph{classical} positive curl (one classical crossing). Let $D_{\mathrm{rig}}$ be the diagram obtained from the same picture
by declaring that crossing to be a \emph{rigid-height vertex} (recall right hand side of Figure~\ref{fig:stuck_diagram}). These two diagrams
are classically equivalent (both represent the unknot under the rigidity-forgetting map), but they behave differently under the stuck bracket normalization.

\smallskip\noindent
\textbf{(i) The classical curl.}
In the classical Kauffman bracket, a positive curl contributes a factor $-A^3$:
\[
\langle D_{\mathrm{cl}}\rangle = -A^3\langle \bigcirc\rangle = -A^3.
\]
The writhe counts only classical crossings, so $w(D_{\mathrm{cl}})=+1$. Hence the normalized
polynomial is
\[
P_{\iota(\text{unknot})}(A,1)
= (-A^3)^{-w(D_{\mathrm{cl}})}\langle D_{\mathrm{cl}}\rangle
= (-A^3)^{-1}(-A^3)=1,
\]
as expected for the classical unknot (here we have written $R=1$ to emphasize reduction to the
classical normalization).

\smallskip\noindent
\textbf{(ii) The rigid curl.}
In $D_{\mathrm{rig}}$ there are no classical crossings at all (the unique crossing is rigid), so there are no $A/B$ choices and there is exactly one state. In that state, the rigid vertex persists and contributes a single factor of $R$. Moreover, the resolved diagram has one underlying component (ignoring the rigid vertex), so $|s|=1$ and $\delta^{|s|-1}=\delta^0=1$. Therefore
\[
\langle D_{\mathrm{rig}}\rangle_R = R.
\]
Since rigid vertices do not contribute to writhe, $w(D_{\mathrm{rig}})=0$, and thus
\[
P_{K^*}(A,R)
= (-A^3)^{-w(D_{\mathrm{rig}})}\langle D_{\mathrm{rig}}\rangle_R
= R.
\]

\smallskip
This computation illustrates that declaring a crossing rigid can change the invariant even when the
underlying classical knot type remains the unknot.
\end{example}

\begin{remark}
The stuck bracket may be viewed as a natural extension of the classical Kauffman bracket to diagrams containing rigid-height vertices. Once the local smoothing weights for classical crossings are fixed, introducing a multiplicative parameter for rigid vertices provides the simplest mechanism for incorporating rigidity into a state-sum framework. In this sense, the stuck bracket plays a role analogous to that of the Kauffman bracket in classical knot theory. This perspective suggests the possibility of developing skein modules \cite{P} adapted to embeddings with rigid vertices, a direction we do not pursue here.
\end{remark}

\paragraph{Comparison with the HOMFLYPT-Type Invariant.}

The state-sum invariant complements the HOMFLYPT-type polynomial constructed earlier. Although both encode rigidity, they arise from distinct algebraic frameworks. The HOMFLYPT-type polynomial extends oriented skein theory by modifying crossing relations to incorporate rigid-height structure. In contrast, the stuck bracket treats rigid vertices as persistent states within the resolution process, allowing their contribution to remain visible throughout the expansion rather than being absorbed into crossing data.

\begin{remark}
The complementary nature of the two constructions is already visible 
in Examples~\ref{curl_stuck_hom} and~\ref{ex:curl_vs_rigidcurl}. 
For the unknot containing a single rigid crossing, the HOMFLYPT-type 
polynomial records rigidity through modified skein relations involving 
classical smoothings, whereas the stuck bracket assigns a persistent 
multiplicative contribution arising directly from the rigid vertex state. 
Thus the former incorporates rigidity into crossing data, while the latter retains rigidity explicitly throughout the state expansion.
\end{remark}


\section{Stuck Knots as Rigid Spatial Graphs}\label{sec:SpatialGraphs}

The presence of stuck crossings places stuck knots naturally within the framework of spatial graph topology. Rather than viewing a stuck crossing as a decorated double point, it is more appropriately interpreted as a rigid vertex whose local structure must be preserved under ambient isotopy. This perspective situates stuck knots within the broader paradigm of topology under local constraints.

Rigid vertex spatial graphs were introduced to study embeddings of graphs in $S^3$ for which the cyclic ordering of edges at each vertex is fixed throughout isotopy.

\begin{definition}
A \emph{rigid spatial graph} is an embedding of a finite graph in $S^3$, considered up to ambient isotopy that preserves the cyclic ordering of edges at every vertex.
\end{definition}

Rigid vertex spatial graphs have been extensively studied, and several invariants have been developed to analyze their embeddings (see for example Yamada \cite{Yamada1989}). Foundational treatments of rigid vertex embeddings may be found in Kauffman \cite{Kauffman1991}.

Under this viewpoint, each stuck crossing corresponds to a rigid $4$-valent vertex equipped with height data specifying which incident strand passes over the other. The height ordering refines the usual rigid vertex structure by introducing directional information analogous to that carried by classical crossings.

\begin{figure}[ht]
\centering
\includegraphics[width=0.7\textwidth]{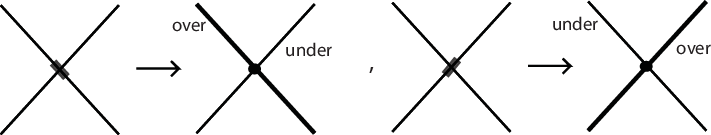}
\caption{A stuck crossing interpreted as a rigid $4$-valent vertex with height ordering.}
\label{fig:rigid_vertex}
\end{figure}

Consequently, every stuck knot naturally determines a rigid spatial graph whose rigid vertices occur precisely at the stuck crossings.

\begin{proposition}
There is a natural correspondence between stuck knots and rigid spatial graphs equipped with height-decorated $4$-valent vertices.
\end{proposition}

\begin{proof}
Given a stuck knot diagram, replace each stuck crossing by a rigid $4$-valent vertex while preserving the height ordering of the incident arcs. Classical crossings remain transverse double points realized by the usual lifting construction. An ambient isotopy of stuck knots preserves rigid crossings together with their height data, and therefore induces an ambient isotopy of the associated spatial graph respecting vertex rigidity. Conversely, any such isotopy of the spatial graph projects to a sequence of admissible diagrammatic moves. Hence the construction is well defined.
\end{proof}

This interpretation shows that rigidity is an intrinsic feature of the spatial embedding rather than a diagrammatic embellishment: the allowable isotopies themselves are restricted by the presence of rigid vertices. Stuck knots may therefore be regarded as constrained spatial embeddings rather than decorated projections.

The additional height structure further distinguishes stuck knots from classical rigid vertex graphs, where no over-under ordering is specified. In this sense, stuck knots occupy an intermediate position between classical knot theory and rigid spatial graph theory, combining features of both while exhibiting rigidity-sensitive phenomena absent from either setting alone.

\begin{remark}
The rigid spatial graph interpretation suggests a complementary combinatorial viewpoint. 
Given a diagram $D$ of a stuck knot, one may associate a graph whose vertices correspond to crossings and whose edges record strand segments connecting crossings without intermediate intersections. Restricting to rigid crossings yields a subgraph that reflects how local constraints are distributed throughout the diagram.

Although this construction is diagram-dependent and not invariant under Reidemeister moves, it provides a natural language for describing regions in which rigidity is concentrated and for organizing the interaction between constrained crossings. Developing a systematic theory relating such combinatorial structures to simplification processes and unsticking phenomena presents an interesting direction for future work.
\end{remark}


\section{Rigidity Release and Unsticking Distance}\label{sec:Unsticking Distance}

Rigid-height vertices act as local constraints on ambient isotopy and may prevent classical simplification moves from occurring. In particular, configurations that would ordinarily cancel via Reidemeister moves can remain obstructed solely due to the presence of rigidity. This observation motivates enlarging the class of admissible diagrammatic operations. We formalize this idea through a relaxed equivalence framework in which rigid crossings may be irreversibly converted into classical crossings. The resulting notion leads naturally to a geometric measure of how strongly rigidity constrains a stuck knot.

\subsection{Relaxed Isotopy}

\begin{definition}
A \emph{relaxed isotopy} between stuck knot diagrams is a finite sequence of diagrams in which each step is either
\begin{enumerate}
\item a stuck isotopy move, or
\item an \emph{unstick move}, converting a rigid crossing into a classical crossing while preserving its over-under structure.
\end{enumerate}
Rigid crossings, once unstuck, remain classical throughout the sequence.
\end{definition}

Relaxed isotopy enlarges the classical move system by permitting the irreversible release of rigidity constraints. Classical Reidemeister moves may therefore occur after sufficient rigidity has been removed.

\begin{figure}[ht]
\centering
\includegraphics[width=0.85\textwidth]{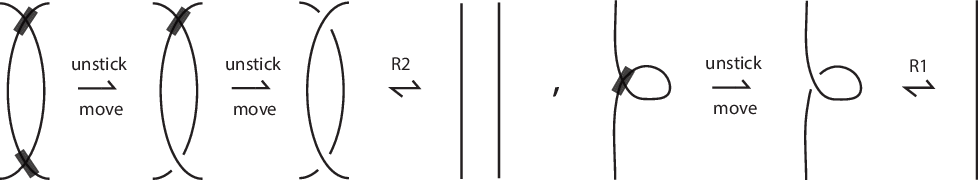}
\caption{Left: A pair of adjacent rigid crossings obstructing a Reidemeister 2 move. After unsticking both vertices, the crossings cancel. Right: A rigid twist that cannot be removed under stuck isotopy. Releasing the rigidity allows a Reidemeister 1 move to eliminate the crossing.}
\label{fig:unstick_RII}
\end{figure}

\begin{remark}
Relaxed isotopy should be viewed as a controlled relaxation of structural constraints rather than a change in the ambient isotopy class of the underlying knot. The purpose of the unstick move is not to alter the classical knot type, but to measure the extent to which rigidity obstructs diagrammatic simplification.
\end{remark}


\subsection{Unsticking Distance}

In this subsection we introduce a quantitative measure capturing how strongly rigidity constrains a stuck knot. Since both stuck isotopy and the unstick move preserve the underlying classical knot type, any sequence of relaxed moves can only relate stuck knots whose classical projections are ambient isotopic. Thus the notion of distance introduced below is meaningful within a fixed classical knot type.

\begin{definition}
Let $K_1^*$ and $K_2^*$ be stuck knot types. The \emph{unsticking distance} $u(K_1^*,K_2^*)$ is defined as the minimal number of unstick moves required in any relaxed isotopy connecting representatives of $K_1^*$ and $K_2^*$. If no such sequence exists, we set $u(K_1^*,K_2^*)=\infty$.
\end{definition}

\begin{remark}
Because the unstick move alters only the rigidity structure while preserving the ambient isotopy class of the underlying embedding, the quantity $u(K_1^*,K_2^*)$ is finite only when $K_1^*$ and $K_2^*$ share the same classical knot type. In particular, the distance from a stuck knot to the trivial stuck knot is finite only when the underlying classical knot is itself the unknot.
\end{remark}

Once selected rigid crossings are released, the resulting diagram contains only classical crossings in the affected regions. Classical Reidemeister moves may then be applied freely, allowing simplifications that were previously obstructed by rigidity.

We first record basic lower and upper bounds.

\begin{proposition}
Let $K^*$ be a stuck knot with $k$ rigid crossings whose underlying classical knot type is the unknot. Then
\[
0 \le u(K^*,\bigcirc) \le k.
\]
\end{proposition}

\begin{proof}
The lower bound is immediate. Unsticking all $k$ rigid crossings produces a classical diagram of the unknot, which can then be reduced to the trivial diagram via classical Reidemeister moves. Hence at most $k$ unsticking operations are required.
\end{proof}

This estimate shows that the unsticking distance is controlled by the number of rigid crossings, while still allowing for the possibility that fewer constraints need to be released when rigidity is concentrated in configurations removable after a small number of unsticking operations. We next observe that unsticking distance behaves subadditively under concatenation of relaxed isotopies.

\begin{proposition}
For stuck knot types $K_1^*, K_2^*, K_3^*$, suppose there exist relaxed isotopies from $K_1^*$ to $K_2^*$ and from $K_2^*$ to $K_3^*$. Then the unsticking distance satisfies
\[
u(K_1^*,K_3^*) \le u(K_1^*,K_2^*) + u(K_2^*,K_3^*).
\]
\end{proposition}

\begin{proof}
Let $\gamma_1$ be a relaxed isotopy from $K_1^*$ to $K_2^*$ realizing $u(K_1^*,K_2^*)$, and let $\gamma_2$ be one from $K_2^*$ to $K_3^*$ realizing $u(K_2^*,K_3^*)$. Performing first a relaxed isotopy from $K_1^*$ to $K_2^*$ and then one from $K_2^*$ to $K_3^*$ yields a relaxed isotopy from $K_1^*$ to $K_3^*$. The number of unstick moves in the concatenated sequence is at most the sum of the two quantities, establishing the inequality.
\end{proof}

\begin{remark}
The unsticking distance therefore behaves like a path-length in the space of stuck knots under rigidity release. Although this suggests a geometric structure on the collection of stuck knots, we do not pursue a metric formulation here.
\end{remark}


\subsection{Rigidity as an Obstruction}\label{subsec:obstruction}

The unsticking distance admits a geometric interpretation in terms of \emph{move obstructions}. In classical knot theory, local configurations may simplify through Reidemeister moves, allowing crossings to cancel or be removed by ambient deformation. In the stuck setting, however, a
rigid-height crossing cannot participate freely in such local simplifications: rigidity may \emph{freeze} an otherwise removable configuration. In this sense, rigid crossings act as \emph{local barriers}
preventing specific Reidemeister moves from being realized until rigidity has been released.

This viewpoint is useful for deriving lower bounds on unsticking distance. Rather than analyzing an entire diagram globally, one may locate local configurations that are classically reducible but are
rigidity-obstructed. Each such configuration forces at least one unstick move before the corresponding classical simplification becomes available.

\medskip

Throughout, when we say that a configuration is \emph{supported in a disk} we mean that there exists an
embedded disk $\Delta$ in the projection plane such that $\Delta$ contains exactly the crossings and arcs
depicted in the configuration, and the diagram meets $\partial\Delta$ in the appropriate number of
transverse boundary points. This standard locality condition ensures that the obstruction is genuinely
local and does not depend on the global geometry of the diagram.


\subsubsection*{Rigid twists and obstructed Reidemeister 1}

A particularly simple obstruction arises from a twist whose unique crossing is rigid. In the classical theory this is eliminated by a Reidemeister 1 move, but in stuck isotopy the rigid crossing cannot be
removed. Once the crossing is unstuck, it becomes a classical crossing and the Reidemeister 1 move becomes admissible (recall right hand side of Figure~\ref{fig:unstick_RII}).

\begin{definition}
A \emph{rigid twist} is a diagrammatic configuration supported in a disk whose underlying classical diagram is a Reidemeister 1 kink and whose unique crossing is rigid.
\end{definition}

\begin{proposition}
If a diagram of a stuck knot contains a rigid twist supported in a disk, then any relaxed isotopy simplifying the diagram must include at least one unstick move supported in that disk. In particular, the presence of a rigid twist forces the inequality
\[
u(K^*,K'^*) \ge 1
\]
for any stuck knot type $K'^*$ represented by a diagram in which that twist has been removed.
\end{proposition}

\begin{proof}
By hypothesis, the disk contains a single crossing whose local underlying classical configuration is a Reidemeister 1 kink. Under stuck isotopy, rigid crossings cannot be removed, so the kink cannot be eliminated by stuck isotopy moves alone. In a relaxed isotopy, the only operation that changes rigidity is an unstick move, and once the crossing is unstuck it becomes classical. At that point the classical Reidemeister 1 move is admissible (it is performed away from rigid crossings) and removes the kink. Therefore at least one unstick move in that disk is necessary to realize the simplification.
\end{proof}

\begin{remark}
This is the minimal obstruction phenomenon: a \emph{single} rigid crossing can block an entire local simplification move. This illustrates that rigidity constrains the \emph{move calculus} rather than merely
contributing to crossing count.
\end{remark}


\subsubsection*{Half-rigid Reidemeister 2 and one-move barriers}

A second basic obstruction occurs when a Reidemeister 2 cancellation would remove a pair of crossings, but rigidity prevents the cancellation even if only one of the two crossings is rigid (recall left hand side of Figure~\ref{fig:unstick_RII}).

\begin{definition}
A \emph{half-rigid Reidemeister 2 configuration} is a local configuration supported in a disk whose underlying classical diagram is a Reidemeister 2 pair, and in which exactly one of the two crossings is
rigid (the other being classical), with the rigid crossing having the same over-under type as required by the classical R2 cancellation.
\end{definition}

\begin{proposition}
If a diagram contains a half-rigid Reidemeister 2 configuration supported in a disk, then any relaxed isotopy eliminating that local pair of crossings requires at least one unstick move supported in that disk.
\end{proposition}

\begin{proof}
Classically, the two crossings cancel by a Reidemeister 2 move. In the stuck setting, the cancellation is inadmissible as long as one of the two crossings is rigid: the local move would remove (or separate) the
rigid-height vertex, which is forbidden under stuck isotopy. In a relaxed isotopy, the only way to make the Reidemeister 2 move available is to first unstick the rigid crossing, converting it into a classical crossing.
After this conversion the local diagram inside the disk is precisely a classical R2 pair, so the Reidemeister 2 move cancels it. Hence at least one unstick move is necessary.
\end{proof}

\begin{remark}
The previous proposition shows that lower bounds for the unsticking distance can arise from local configurations in which \emph{only one} rigid crossing obstructs a classical cancellation. This is sharper than the situation where both crossings in an R2 pair are rigid, since one unstick move can already unlock the cancellation.
\end{remark}


\subsubsection*{Independent barriers and lower bounds}

The obstruction viewpoint becomes particularly effective when one can find many independent local
barriers supported in disjoint regions.

\begin{definition}
A collection of local configurations in a diagram is \emph{pairwise disjoint} if each configuration is supported
in an embedded disk and these disks are pairwise disjoint.
\end{definition}

\begin{proposition}\label{prop:independent_barriers}
Suppose a diagram of a stuck knot contains $m$ pairwise disjoint rigidity-obstructed Reidemeister configurations, each of which becomes removable by a classical Reidemeister move after at least one unstick move inside its supporting disk. Then
\[
u(K^*,K'^*) \ge m
\]
for any stuck knot type $K'^*$ represented by a diagram in which all $m$ configurations have been removed. In particular, if the underlying classical knot is the unknot and $K'^*=\bigcirc$, then
\[
u(K^*,\bigcirc)\ge m.
\]
\end{proposition}

\begin{proof}
Fix supporting disks $\Delta_1,\dots,\Delta_m$ for the $m$ configurations, with $\Delta_i\cap\Delta_j=\emptyset$ for $i\neq j$. By assumption, for each $i$ the local configuration inside $\Delta_i$ cannot be removed by stuck isotopy alone; it becomes removable only after performing at least one unstick move supported in $\Delta_i$. Because the disks are disjoint, an unstick move performed in $\Delta_i$ cannot affect the rigidity structure or local move admissibility inside $\Delta_j$ for $j\neq i$. Hence the required rigidity releases are independent:
any relaxed isotopy that removes all $m$ configurations must contain at least one unstick move in each disk. Therefore at least $m$ unstick moves are required in total.
\end{proof}

\begin{figure}[ht]
\centering
\includegraphics[width=0.52\textwidth]{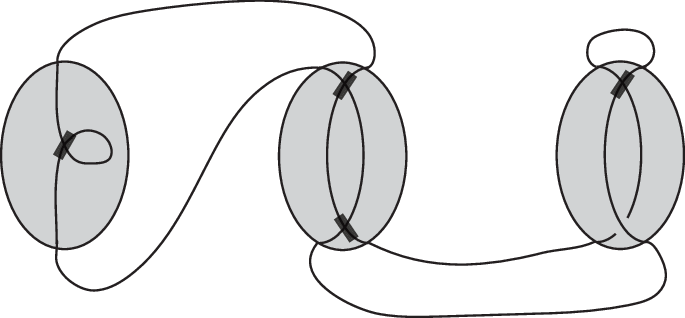}
\caption{Pairwise disjoint rigidity barriers supported in disjoint disks. Each barrier requires at least one unstick move before a local Reidemeister simplification becomes admissible.}
\label{fig:barriers}
\end{figure}

Proposition~\ref{prop:independent_barriers} expresses a general principle: rigidity creates \emph{local} geometric obstructions to simplification. Locating families of barriers supported in disjoint disks yields systematic lower bounds for unsticking distance. More broadly, this suggests viewing rigid crossings as structural constraints that partition a diagram into regions of constrained behavior. Understanding how such barriers interact (whether they reinforce, propagate or
neutralize one another) is a natural direction for future investigation.

\subsection{Examples of Unsticking Distance}

We conclude by illustrating the unsticking distance through several elementary configurations that exhibit progressively stronger forms of rigidity obstruction.

\begin{example}
The presence of a single rigid crossing prevents a Reidemeister 1 move that would otherwise eliminate the twist. Once the crossing is unstuck, the classical reduction becomes available, and removes the crossing immediately. Hence this configuration has unsticking distance one.
\end{example}

\begin{example}
Consider a diagram of the unknot containing a Reidemeister 2 configuration in which one crossing is rigid and the other is classical. Although the crossings would ordinarily cancel, the rigid-height vertex prevents the strands from separating, so stuck isotopy alone cannot remove the pair. After a single unstick move at the rigid crossing, the configuration becomes a classical Reidemeister 2 pair and cancels immediately. Hence the unsticking distance to the trivial stuck knot is one.
\end{example}

\begin{example}
Consider a diagram of the unknot containing several rigid twists supported in pairwise disjoint disks. Each twist requires one unstick move before a Reidemeister 1 reduction becomes available. Because the configurations are disjoint, the required unsticking operations are independent. If there are $m$ such twists, then any relaxed isotopy simplifying the diagram must include at least $m$ unstick moves, and therefore
\[
u(K^*,\bigcirc) \ge m.
\]
\end{example}

These examples highlight that the unsticking distance depends not only on the number of rigid crossings but also on their geometric arrangement. Even diagrams representing the classical unknot may exhibit large unsticking distance when rigidity is distributed across independent local barriers. This perspective reinforces the interpretation of rigidity as a structural obstruction to isotopy and suggests that constrained embeddings possess a notion of complexity distinct from classical measures such as crossing number or unknotting number.




\begin{thebibliography}{99}

\bibitem{Bataineh2020}
K.~Bataineh,
\newblock Stuck knots,
\newblock {\em Symmetry}, \textbf{12} (2020), no.~9, 1558.

\bibitem{BondarenkoCenicerosElhamdadiJones2025}
E.~Bondarenko, J.~Ceniceros, M.~Elhamdadi, and B.~Jones,
\newblock Generalized quandle polynomials and their applications to stuquandles, stuck links, and RNA folding,
\newblock {\em Open Math.} \textbf{23} (2025), 20250151.

\bibitem{CenicerosElhamdadiKomissarLahrani2024}
J.~Ceniceros, M.~Elhamdadi, J.~Komissar, and H.~Lahrani,
\newblock RNA foldings and stuck knots,
\newblock {\em Communications of the Korean Mathematical Society} \textbf{39} (2024), no.~1, 223--245.

\bibitem{CenicerosElhamdadiMagillRosario2023}
J.~Ceniceros, M.~Elhamdadi, B.~Magill, and G.~Rosario,
\newblock RNA foldings, oriented stuck knots and state sum invariants,
\newblock {\em J. Math. Phys.} \textbf{64} (2023), 031702.

\bibitem{Diamantis2024}
I.~Diamantis, S.~Lambropoulou, and S.~Mahmoudi,
\newblock From annular to toroidal pseudo knots,
\newblock {\em Symmetry} \textbf{16} (2024), no.~10, 1360.

\bibitem{DiamantisKauffmanLambropoulou2025}
I.~Diamantis, L.~H.~Kauffman, and S.~Lambropoulou,
\newblock Topology and algebra of bonded knots and braids,
\newblock {\em Mathematics} \textbf{13} (2025), no.~20, 3260.

\bibitem{Hanaki2010}
R.~Hanaki,
\newblock Pseudo diagrams of knots, links and spatial graphs,
\newblock {\em Osaka J. Math.} \textbf{47} (2010), 863--883.

\bibitem{HenrichKauffman2017}
A.~Henrich and L.H.~Kauffman,
\newblock Tangle insertion invariants for pseudoknots, singular knots, and rigid vertex spatial graphs,
\newblock {\em Knots, Links, Spatial Graphs, and Algebraic Invariants} (E. Flapan et al., eds.), 
\newblock Contemporary Mathematics, vol. 689, Amer. Math. Soc., 2017, 61--83.

\bibitem{HOMFLY}
P.~Freyd, D.~Yetter, J.~Hoste, W.~B.~R.~Lickorish, K.~Millett, and A.~Ocneanu,
\newblock A new polynomial invariant of knots and links,
\newblock {\em Bull. Amer. Math. Soc.} \textbf{12} (1985), 239--246.

\bibitem{Jones1985}
V. F. R. Jones,
\newblock A polynomial invariant for knots via von Neumann algebras,
\newblock {\em Bull. Amer. Math. Soc.} \textbf{12} (1985), 103--111.

\bibitem{Kauffman1987}
L.~H.~Kauffman,
\newblock State models and the Jones polynomial,
\newblock {\em Topology} 1987, \textbf{26}, 395--407.

\bibitem{Kauffman1989}
L.~H.~Kauffman,
\newblock Invariants of graphs in three-space,
\newblock {\em Trans. Amer. Math. Soc.} 1989, \textbf{311}, 697--710.

\bibitem{Kauffman1991}
L.~H.~Kauffman,
\newblock Knots and Physics,
\newblock {\em World Scientific}, Singapore, 1991.

\bibitem{Kauffman1999}
L.H.~Kauffman,
\newblock Virtual knot theory,
\newblock {\em European Journal of Combinatorics}, {\bf 20}(7), 1999, 663--691.

\bibitem{P}
J.~Przytycki,
\newblock Skein modules of 3-manifolds,
\newblock \emph{Bull. Pol. Acad. Sci.: Math.}, {\bf 39, 1-2} (1991), 91--100.

\bibitem{Rolfsen}
D.~Rolfsen,
\newblock Knots and Links,
\newblock \emph{AMS Chelsea Publishing}, 2003.

\bibitem{Vassiliev1990}
V.A.~Vassiliev,
\newblock Cohomology of knot spaces,
\newblock Theory of Singularities and its Applications (V.I. Arnold, ed.),
\newblock {\em  Adv. Soviet Math.}, vol. 1, Amer. Math. Soc., Providence, RI, 1990, 23--69.

\bibitem{Yamada1989}
S.~Yamada,
\newblock An invariant of spatial graphs,
\newblock {\em J. Graph Theory} \textbf{13} (1989), No.~5, 537--551.



\end{thebibliography}
\end{document}